\documentclass[11pt]{article}
\usepackage{amsmath, amssymb, amsfonts, amsthm, authblk, color, graphicx, enumitem, mathrsfs, indentfirst}
\usepackage{hyperref, cleveref}
\hypersetup{colorlinks=true, linkcolor=blue, filecolor=magenta, urlcolor=cyan}
\usepackage[textwidth=6in,textheight=9in]{geometry}
\allowdisplaybreaks

\usepackage{algorithm}
\usepackage{algorithmic}
\usepackage{subcaption}
\usepackage{chngcntr}
\counterwithin{table}{section}
\newtheorem{theorem}{Theorem}[section]
\newtheorem{lemma}{Lemma}[section]
\newtheorem{example}{Example}[section]
\newtheorem{assumption}{Assumption}

\numberwithin{equation}{section}

\title{An efficient iteration method to reconstruct the drift term from the final measurement}

\author[1]{Dakang Cen\thanks{cendakang@163.com}}
\author[1]{Wenlong Zhang\thanks{zhangwl@sustech.edu.cn}}
\author[2,3]{Zhidong Zhang\thanks{zhangzhidong@mail.sysu.edu.cn}}
\affil[1]{Department of Mathematics, Southern University of Science and Technology, Shenzhen, 518055, China}
\affil[2]{School of Mathematics (Zhuhai), Sun Yat-sen University, Zhuhai 519082, Guangdong, China}
\affil[3]{Guangdong Province Key Laboratory of Computational Science, Sun Yat-sen
University, Guangzhou 510000, Guangdong, China}

\begin{document}
\maketitle

\abstract{This work investigates the inverse drift problem in the one-dimensional parabolic equation with the final time data. The authors construct an operator first, whose fixed points are the unknown drift, and then apply it to prove the uniqueness. The proof of uniqueness contains an iteration converging to the drift, which inspires the numerical algorithm. To handle the ill-posedness of the inverse problem, the authors add the mollification on the data first in the iterative algorithm, and then provide some numerical results.\\

Keywords: inverse drift problem, uniqueness, monoticity, iterative algorithm.\\

AMS classification codes: 35K10, 35R30, 65M32
}
\section{Introduction.}

In this work, for the parabolic equation $(\partial_t-\Delta+q(x)\cdot\nabla+p(x))u=f$, we aim to recover the drift term $q$. The drift $q$ could be called as the first order term (since it is the coefficient of the gradient of solution), also we could see that this inverse problem is nonlinear. 
The nonlinearity and the high order make the inverse drift problem be challenging. The authors would use the method of monotone operator to solve the inverse problem, which may be technical. This article considers the one-dimensional case of this inverse problem, which could be regarded as a start of such approach. Also, we assume that the potential term $p$, the source $f$ and the boundary condition are controllable to ensure the approach works.      

The one-dimensional mathematical model is given as follows: 
\begin{equation}\label{PDE}
 \begin{cases}
  \begin{aligned}
   (\partial_t-\partial_x^2+q(x)\partial_x+C_p)u(x,t)&=f(x), &&(x,t)\in(0,1)\times(0,T];\\
   u_x(0,t)&=b_1,&&t\in(0,T];\\
   u_x(1,t)&=b_2(t),&&t\in(0,T];\\
   u(x,0)&=v(x),&&x\in(0,1).\\
  \end{aligned}
 \end{cases}
\end{equation}
The source $f$, the potential constant $C_p$, the initial condition $v(x)$ and the boundary conditions $b_1,\ b_2$ are given. Note that the left boundary $b_1$ is constant. We aim to recover the unknown drift term $q(x)$. The measurement is the final time data 
\begin{equation}\label{data}
 g(x):=u(x,T).
\end{equation}
So the interested inverse problem can be described as  
\begin{equation}\label{inverse_problem}
 \text{to use the data}\ g\ \text{to recover the drift}\ q. 
\end{equation}

To solve the nonlinear inverse problem as \eqref{inverse_problem}, the monotone operator method could be an appropriate choice. The first step of such approach is to construct an operator from equation \eqref{PDE} and data \eqref{data}, whose fixed point is the solution of the inverse problem. To make sure the operator to be well-defined, we choose the Neumann boundary. Then some properties of the operator are proved, like monotonicity and the uniqueness of the fixed point (see Lemma \ref{lemma_monotone} and Theorem \ref{theorem_unique}, respectively). One of the advantages of this method is that the proof of uniqueness of the inverse problem contains an iteration converging to the unknown drift, which could be used in the numerical reconstruction. See the proof of Theorem \ref{theorem_unique} for details.

Drift terms have rich context and significance in physics models.
In terms of transport phenomena, the coupling effect of drift and diffusion in pipe flow was investigated in \cite{Taylor:1953}. For finance problems, the drift term represents the risk-free interest rate in Black-Scholes model \cite{BS:1973}. In chemotaxis models, it describes the orientation of the chemical signal \cite{KS:1971}. In addition, drift describes directional motion under the influence of potential energy and migration speed of the population for the Fokker-Planck equation \cite{Risken:1989} and the population diffusion model \cite{GurtinMacCamy:1977}, respectively.
On the other hand, the reconstruction of drift terms is also of interest in application scenarios. However, due to the nonlinearity and the high order property, the research on inverse drift problems seems relatively scarce. 
 The inverse problem of determining the parameter of the expected growth rate is considered for the backward parabolic equation \cite{KorolevKuboYagola:2012}. Subsequently, the uniqueness of the solution to an inverse problem with respect to the real drift is proved by applying microlocal analysis in \cite{DoiOta:2018}.
 
For the literature that using the monotone operator method to solve inverse problems, we collect several works as follows.  \cite{Jones:1962, Jones:1963} are the early works of such method. The author builds an integral equation whose fixed point is the unknown coefficient. Then the uniqueness and stability results of the inverse problem are given. \cite{Zhang:2016,Zhang:2017} recover the time-dependent second order term of the fractional diffusion equation, proving the uniqueness result. The authors in \cite{ZhangZhangZhou:2022} consider the inverse potential problem of the parabolic equation. They prove the uniqueness and conditional stability of the inverse problem, and provide several numerical results and the corresponding error estimate. \cite{Duchateau:1995} investigates the coefficient-to-data mappings associated with unknown coefficient inverse problems for nonlinear parabolic partial differential equations. It shows that the mapping is monotone and invertible.

The manuscript is structured as follows. In section \ref{section_pre}, we collect several preliminary knowledge, as the maximum principles of elliptic and parabolic equations, and the positivity results of the solution. In section \ref{section_unique}, we prove the uniqueness theorem--Theorem \ref{theorem_unique}, which is the main result of this work. One monotone operator is constructed for the proof. In section \ref{section_num}, the authors consider the numerical reconstructions and provide several numerical results. The concluding remarks and future works are discussed in section \ref{section_con}.

\section{Preliminaries.}\label{section_pre}
Some preliminary results would be provided in this section. 
\subsection{Maximum principles.}
The next two lemmas are the maximum principles of one-dimensional elliptic and parabolic equations, respectively, which would be used in the later proofs. 
\begin{lemma}[\cite{Evans:1998}]\label{lemma_max_elliptic}
We set 
$$Lu=(-\partial_x^2+a(x)\partial_x+c(x))u, \quad c\ge 0$$ and $a,\ c$ are smooth enough.   
\begin{itemize}
 \item [(a)] If $Lu\le 0$ in $(0,1)$ and $u$ attains a nonnegative maximum over $[0,1]$ at an interior point, then $u$ is a constant in $(0,1)$.
 \item [(b)] If $Lu\ge 0$ in $(0,1)$ and $u$ attains a nonpositive minimum over $[0,1]$ at an interior point, then $u$ is a constant in $(0,1)$.
\end{itemize}
\end{lemma}

\begin{lemma}[\cite{Evans:1998}]\label{lemma_max_parabolic}
 We use the same elliptic operator $L$ in Lemma \ref{lemma_max_elliptic}. 
 \begin{itemize}
  \item [(a)] If $u_t+Lu\le 0$ in $(0,1)\times (0,T]$ and $u$ attains a nonnegative maximum over $[0,1]\times[0,T]$ at a point $(x_0,t_0)\in (0,1)\times(0,T]$, then 
  $u$ is a constant on $(0,1)\times(0,t_0]$.
  \item [(b)] If $u_t+Lu\ge 0$ in $(0,1)\times (0,T]$ and $u$ attains a nonpositive minimum over $[0,1]\times[0,T]$ at a point $(x_0,t_0)\in (0,1)\times(0,T]$, then 
  $u$ is a constant on $(0,1)\times(0,t_0]$.
 \end{itemize}
\end{lemma}

\subsection{Assumptions and the positivity result. }
Throughout the manuscript, we set the next assumptions be valid.
\begin{assumption}\label{assumption} 
The source $f$, initial condition $v$, boundary conditions $b_1,\ b_2(t)$, the potential constant $C_p$ and the exact drift $q$ are sufficiently smooth and satisfy the compatible condition. We give other assumptions as follows.  
 \begin{itemize}
 \item [(a)] $q\in C^1(0,1)$ with $\|q\|_{C^1(0,1)}<M$, where $M>0$ is a given sufficiently large constant.  
  \item [(b)] The potential constant $C_p$ is strictly larger than the upper bound $M$ in $(a)$. 
  \item [(c)] The left boundary constant $b_1$ is strictly positive. 
  \item [(d)] $b_2$ and $b'_2$ are strictly positive on $(0,T]$.   
  \item [(e)] $v\in C^3(0,1)$, $v'\ge 0$ on $(0,1)$, and $C_v:=\max\{\|v\|_{C^j(0,1)}:j=0,1,2,3\}<\infty.$
  \item [(f)] $f\ge (1+M+C_p)C_v$ and $f'\ge (1+2M+C_p)C_v$ on $(0,1)$.
 \end{itemize}
\end{assumption}

With Assumption \ref{assumption}, we could state the next positivity results and give the proofs. 

\begin{lemma}\label{lemma_positivity}
 Under Assumption \ref{assumption}, for model \eqref{PDE} we have the next results. 
\begin{itemize}
 \item [(a)] $u_x> 0$ on $(0,1)\times(0,T]$.
 \item [(b)] $\exists m>0$ such that $u_x(\cdot,T)>m$ on $(0,1)$. 
 \item [(c)] $u_t\ge 0$ on $[0,1]\times[0,T]$.
 \item [(d)] $u_{xt}\ge 0$ on $[0,1]\times[0,T]$.
\end{itemize}
\end{lemma}
\begin{proof}
 For $(a)$, with $w=u_x$ we have that  
 \begin{equation}\label{model_ux}
 \begin{cases}
  \begin{aligned}
   (\partial_t-\partial_x^2)w+qw_x+(q'+C_p)w&=f'\ge 0, &&(x,t)\in(0,1)\times(0,T];\\
   w(0,t)&=b_1>0,&&t\in(0,T];\\
   w(1,t)&=b_2(t)>0,&&t\in(0,T];\\
   w(x,0)&=v'\ge0,&&x\in(0,1).\\
  \end{aligned}
 \end{cases}
\end{equation}
From Assumption \ref{assumption} we have that the potential $q+C_p$ of the above system is positive. Then by Lemma \ref{lemma_max_parabolic}, it holds that $w=u_x> 0$ on $(0,1)\times(0,T]$.

For $(b)$, the conclusion $(a)$ and the boundary conditions of model \eqref{model_ux} give that $u_x(\cdot,T)>0$ on $[0,1]$. This and the continuity of $u_x(\cdot,T)$ (the continuity could be ensured by Lemma \ref{lemma_smooth}) gives the desired result.

For $(c)$, by setting $w=u_t$, we have 
 \begin{equation}\label{model_ut}
 \begin{cases}
  \begin{aligned}
   (\partial_t-\partial_x^2+q\partial_x+C_p)w&=0, &&(x,t)\in(0,1)\times(0,T];\\
   w_x(0,t)&=0,&&t\in(0,T];\\
   w_x(1,t)&=b'_2(t)>0,&&t\in(0,T];\\
   w(x,0)&=f+v''-qv'-C_pv\ge 0,&&x\in(0,1).\\
  \end{aligned}
 \end{cases}
\end{equation}
Assume that 
$$w(x_0,t_0):=\min\{w(x,t):(x,t)\in[0,1]\times[0,T]\}<0.$$
Lemma \ref{lemma_max_parabolic} gives that $(x_0,t_0)\notin (0,1)\times(0,T]$. 
From the initial condition of the above model, it gives that $t_0\ne 0$ and $x_0\in \{0,1\}$. For the case of $x_0=1$, with the boundary condition $w_x(1,t_0)> 0$,  there exists $x_1\in (0,1)$ such that $w(x_1,t_0)< w(1,t_0)$. This contradicts with the minimal assumption. For the case of $x_0=0$, the minimal assumption gives that $w_t(0,t_0)\le 0$; the boundary condition $w_x(0,t_0)=0$ with the minimal assumption yields that $u_xx(0,t_0)\ge 0$; the boundary condition $w_x(0,t_0)=0$ gives $q(0)w_x(0,t_0)=0$; the facts $C_p>0$ and $w(0,t_0)<0$ lead to $C_pw(0,t_0)<0$. Hence we have that 
$$(\partial_t-\partial_x^2+q\partial_x+C_p)w(0,t_0)<0,$$
which contradicts with model \eqref{model_ut}. 
So we have $w=u_t\ge 0$ on $[0,1]\times[0,T]$.

For $(d)$, the model of $w=u_{xt}$ could be written as 
\begin{equation*}
 \begin{cases}
  \begin{aligned}
   (\partial_t-\partial_x^2)w+qw_x+(q'+C_p)w&=0, &&(x,t)\in(0,1)\times(0,T];\\
   w(0,t)&=0,&&t\in(0,T];\\
   w(1,t)&=b'_2(t)>0,&&t\in(0,T];\\
   w(x,0)&=f'+v'''-qv''-(q'+C_p)v',&&x\in(0,1).\\
  \end{aligned}
 \end{cases}
\end{equation*}
With the fact that $q'+C_p>0$ and Lemma \ref{lemma_max_parabolic}, we could deduce that $w=u_{xt}\ge 0$ on $[0,1]\times[0,T]$. The proof is complete. 
\end{proof}

\subsection{Regularities.}

Here we give the lemmas concerning the regularity results, which would be used in the proof of the uniqueness for the iterative method. The next two lemmas contain the regularities about $\|u(\cdot,T)\|_{L^2(0,1)}$ and $\|\partial ^2_{xt}u\|_{L^\infty((0,1)\times(0,T))}$, respectively.  
\begin{lemma}\label{lemma_regularity_parabolic}
Under Assumption \ref{assumption}, for the following system
\begin{equation*}
 \begin{cases}
  \begin{aligned}
   (\partial_t-\partial_x^2)w(x,t)+ q(x)w_x(x,t)+C_pw(x,t)&=f_1(x,t), &&(x,t)\in(0,1)\times(0,T];\\
   w_x(0,t)&=0,&&t\in(0,T];\\
   w_x(1,t)&=0,&&t\in(0,T];\\
   w(x,0)&=f_2(x),&&x\in(0,1),\\
  \end{aligned}
 \end{cases}
\end{equation*}
one has
\begin{equation}
\begin{aligned}\nonumber
\|w(\cdot,T)\|_{L^2(0,1)}\leq C\|f_2\|_{L^2(0,1)}+ C\|f_1\|_{L^2((0,1)\times(0,T))}.   
\end{aligned}
\end{equation}
\end{lemma}

\begin{proof}
Setting 
$$E(t):=\frac{1}{2}\int_0^1w(x,t)^2dx,$$ 
we have 
\begin{align*}
\frac{dE}{dt}
&=\int_0^1w(x,t)w_t(x,t)dx\\
&=\int_0^1w(\partial_x^2w-q(x)w_x-C_pw+f_1)dx.
\end{align*} 
Also for 
$$\int_0^1w\partial_x^2wdx\ \text{and}\ \int_0^1wq(x)\partial_xwdx,$$ 
it holds that
\begin{align*}
\int_0^1w\partial_x^2wdx=(w\partial_xw)|_0^1-\int_0^1(\partial_xw)^2dx=-\int_0^1(\partial_xw)^2dx,
\end{align*} 
and
\begin{align*}
-\int_0^1wq(x)\partial_xwdx
&\leq \sqrt{\int_0^1w^2q(x)^2dx}\sqrt{\int_0^1(\partial_xw)^2dx}\\
&\leq \frac14\int_0^1w^2q(x)^2dx+\int_0^1(\partial_xw)^2dx,
\end{align*} 
respectively. Therefore, we could conclude that 
\begin{align*}
\frac{dE}{dt}
&\leq \frac14\int_0^1w^2q(x)^2dx-C_p\int_0^1w^2dx+\int_0^1wf_1dx\leq C\int_0^1w^2dx+\int_0^1wf_1dx,
\end{align*} 
which together with Gronwall inequality completes the proof.
\end{proof}

\begin{lemma}\label{lemma_regularity}
Under Assumption \ref{assumption}, for model \eqref{PDE} we have the next regularities for  $\partial^2_{xt}u$:
\begin{align*}
 \|\partial ^2_{xt}u\|_{L^\infty((0,1)\times(0,T))}&\le \max\bigg\{\|f'\|_{L^\infty(0,1)},\|b_2'\|_{L^\infty(0,T)}\bigg\}.
\end{align*}
\end{lemma}
\begin{proof}
From the proof of Lemma \ref{lemma_positivity}, the model for $w=u_{xt}$ could be given as 
\begin{equation*}
 \begin{cases}
  \begin{aligned}
   (\partial_t-\partial_x^2)w+qw_x+(q'+C_p)w&=0, &&(x,t)\in(0,1)\times(0,T];\\
   w(1,t)&=0,&&t\in(0,T];\\
   w(1,t)&=b'_2(t)>0,&&t\in(0,T];\\
   w(x,0)&=f'\ge 0,&&x\in(0,1).\\
  \end{aligned}
 \end{cases}
\end{equation*}
We assume that 
$$w(x_0,t_0)=\max_{[0,1]\times[0,T]}\{w(x,t)\}.$$
From Lemma \ref{lemma_max_parabolic}, we know that the maximum value is determined by the initial and boundary conditions $w(x,0)$, $w(0,t)$ and $w(1,t)$. The proof is complete.
\end{proof}

The following lemma ensures the smoothness of solution $u$ under the weaker assumption that the drift $q$ only belongs to $L^\infty(0,1)$.  We introduce the so-called $t$-anisotropic Sobolev space 
$$\|u\|_{W_p^{2k,k}((0,1)\times(0,T))}:=\bigg(\int\int_{(0,1)\times(0,T)}\sum_{|\alpha|+2r\leq2k}|D^\alpha D_t^ru|^pdxdt\bigg)^{1/p},$$ 
where $k$ is a nonnegative integer and $1\le p \le +\infty$. Following the idea in \cite[subsection 9.2.2]{Wu:2006}, we have Lemma \ref{lemma_smooth}.
\begin{lemma}\label{lemma_smooth}
In Assumption \ref{assumption}, we use the weaker condition that $q\in{L^\infty}(0,1)$. Then for model
\eqref{PDE}, 
one has
\begin{equation}
\begin{aligned}\nonumber
\|u\|_{W_p^{2,1}((0,1)\times(0,T))}\leq C,~~p>1.   
\end{aligned}
\end{equation}
\end{lemma}

\section{Uniqueness of inverse problem \eqref{inverse_problem}.}\label{section_unique}
In this section we would use the monotone operator method to solve inverse problem \eqref{inverse_problem}.

\subsection{Operator $K$ and the equivalence.}
From equation \eqref{PDE} and data \eqref{data}, we give the definition of the operator $K$ as follows:  
\begin{equation}\label{operator}
 K\psi=\frac{f(x)-\partial_tu(x,T;\psi)+g''(x)-C_pg(x)}{g'(x)},\ x\in[0,1],
\end{equation}
with domain 
\begin{equation}\label{domain}
 \mathcal D=\{\psi\in C^1([0,1]):\psi\le [f(x)+g''(x)-C_pg(x)]/g'(x)\}.
\end{equation}
Here the notation $u(x,t;\psi)$ means the solution $u$ of equation \eqref{PDE} with drift term $\psi$. Lemma \ref{lemma_positivity} ensures the strict positivity of the dominator $g'(x)$, which leads to the well-definedness of operator $K$. 

With the operator $K$, next we need to prove the equivalence between the fixed point of $K$ and the solution of inverse problem \eqref{inverse_problem}. 
\begin{lemma}\label{lemma_equivalence}
 Fixing $q$ in the domain $\mathcal D$ defined in \eqref{domain}, the next two statements are equivalent.
\begin{itemize}
\item[(a)] $q$ is the a fixed point of $K$; 
\item[(b)] $u(x,T;q)=g(x)$, where $u(\cdot,\cdot;q)$ is the solution of \eqref{PDE} with drift $q$ and $g(x)$ is given by data \eqref{data}.
\end{itemize}

\end{lemma}
\begin{proof}
 Statement $(a)$ is obvious from equation \eqref{PDE} and data \eqref{data}. 
 
 For statement $(b)$, from the formulations of equation \eqref{PDE} and operator \eqref{operator}, $w(x):=u(x,T;q)-g(x)$ satisfies 
 \begin{equation*}
 \begin{cases}
  \begin{aligned}
   -w''(x)+q(x)w'(x)+C_pw(x)&=0,\quad x\in(0,1);\\
   w'(0)&=0;\\
   w'(1)&=0.
  \end{aligned}
 \end{cases}
\end{equation*}
Assume that 
$$w(x_0)=\max_{[0,1]}\{w(x)\}>0.$$
Lemma \ref{lemma_max_elliptic} gives that $x_0\in\{0,1\}$. For the case of $x_0=1$, the maximal property and the boundary condition $w'(1)=0$ ensure that $w''(1)\le 0$, which leads to 
$$-w''(1)+q(1)w'(1)+C_pw(1)>0.$$
This contradicts with the model for $w$. The contradiction for $x_0=0$ could be deduced analogously. So we have $w\le 0$ on $[0,1]$. 

Similarly, if we assume 
$$w(x_0)=\min_{[0,1]}\{w(x)\}<0,$$ 
the contradiction could be generated following the above arguments, which leads to $w\ge 0$ on $[0,1]$. Hence we have $w=u(x,T;q)-g(x)\equiv 0$ on $[0,1]$ and complete the proof. 
\end{proof}

\subsection{Monotonicity.}
In this subsection, we would prove that the operator $K$ in \eqref{operator} is a monotone operator, which is contained by the lemma below. 
\begin{lemma}\label{lemma_monotone}
 With $K$ in \eqref{operator} and $\mathcal D$ in \eqref{domain}, given $q_1, q_2\in \mathcal D$, the result $q_1\le q_2$  implies that $Kq_1\le Kq_2$ in $[0,1]$.
\end{lemma}
\begin{proof}
 Setting $w=\partial_tu(x,t;q_1)-\partial_tu(x,t;q_2)$, from \eqref{PDE}, Assumption \ref{assumption} and Lemma \ref{lemma_positivity} we have that 
 \begin{equation*}
 \begin{cases}
  \begin{aligned}
   (\partial_t-\partial_x^2+q_1\partial_x+C_p)w(x,t)
   &=(q_2-q_1)\partial_{xt}^2 u(x,t;q_2)\ge0, &&(x,t)\in(0,1)\times(0,T];\\
   w_x(0,t)&=0,&&t\in(0,T];\\
   w_x(1,t)&=0,&&t\in(0,T];\\
   w(x,0)&= (q_2-q_1)v'\ge 0,&&x\in(0,1),\\
  \end{aligned}
 \end{cases}
\end{equation*}
We assume that 
$$w(x_0,t_0)=\min_{[0,1]\times[0,T]}\{ w(x,t)\}<0.$$
The initial condition with Lemma \ref{lemma_max_parabolic} yields that $(x_0,t_0)\in \{0,1\}\times(0,T]$. For the case of $x_0=1$, the minimal property gives that 
$\partial_tw(x_0,t_0)\le 0$ and $w_{xx}(x_0,t_0)\ge 0$, which together with $w_x(x_0,t_0)=0$ and $w(x_0,t_0)<0$ leads to 
$$(\partial_t-\partial_x^2+q_1\partial_x+C_p)w(x_0,t_0)<0.$$ 
This contradicts with the nonnegativity of the source $\partial_{xt}^2 u(x,t;q_2)(q_2-q_1)$. The proof for the case $x_0$ could be given similarly. Hence we have $w\ge 0$ on $[0,1]\times[0,T]$, which gives $Kq_1\le Kq_2$ and completes the proof.  
\end{proof}

\subsection{The iteration from upper bound and the uniqueness theorem.}

In this subsection we would complete the proof of the uniqueness theorem. Firstly, we prove the uniqueness result under monotone relation, which is contained in the next lemma. \begin{lemma}\label{lemma_uniqueness}
With $K$ in \eqref{operator} and $\mathcal D$ in \eqref{domain}, if $q_1,q_2$ are both the fixed points of $K$ in $\mathcal D$ and satisfy $q_1\le q_2$, then $q_1=q_2$.
\end{lemma}
\begin{proof}
 With $w=u(x,t;q_1)-u(x,t;q_2)$, we have 
 \begin{equation*}
 \begin{cases}
  \begin{aligned}
   (\partial_t-\partial_x^2+q_1\partial_x+C_p)w(x,t)&=(q_2-q_1)u_x(x,t;q_2), &&(x,t)\in(0,1)\times(0,T];\\
   w_x(0,t)&=0,&&t\in(0,T];\\
   w_x(1,t)&=0,&&t\in(0,T];\\
   w(x,0)&=w(x,T)=0,&&x\in(0,1).\\
  \end{aligned}
 \end{cases}
\end{equation*}
The equality $w(x,T)=0$ comes from Lemma \ref{lemma_equivalence}. The proof of Lemma \ref{lemma_monotone} ensures the nonnegativity of $\partial_t w$. This result together with $w(x,0)=w(x,T)=0$ yields that $w\equiv 0$, which leads to $u_x(x,t;q_2)(q_2-q_1)\equiv 0$. Then by the positivity result $u_x(\cdot, T;q_2)>0$, which is ensured by Lemma \ref{lemma_positivity}, we have $q_1=q_2$ and complete the proof. 
\end{proof}

Next we define the iteration of $K$ from upper bound of $\mathcal D$ as follows. 
\begin{equation}\label{iteration}
 q_0=[f(x)+g''(x)-C_pg(x)]/g'(x),\quad q_{n+1}=Kq_n,\quad n=0,1,\cdots.
\end{equation}
The above iteration would be used in the proof of uniqueness. 

Now we could complete the proof of the main result--Theorem \ref{theorem_unique}. 
\begin{theorem}\label{theorem_unique}
 With Assumption \ref{assumption}, $K$ in \eqref{operator} and $\mathcal D$ in \eqref{domain}, if $q\in \mathcal D$ is a fixed point of $K$, then the sequence $\{q_n\}_{n=0}^\infty$ generated by iteration \eqref{iteration} would converge decreasingly to it. Moreover, this ensures the uniqueness of inverse problem \eqref{inverse_problem} in $\mathcal D$. 
\end{theorem}

\begin{proof}
 Given $q\in \mathcal D$ be the solution of inverse problem \eqref{inverse_problem}, it should be one fixed point of $K$. From \eqref{iteration} we see that the initial $q_0\in\mathcal D$ is the upper bound of the domain $\mathcal D$. Then it gives that $q\le q_0$, which leads to $q=Kq\le Kq_0=q_1$ by Lemma \ref{lemma_monotone}. Also, from the nonnegativity of $\partial_t u(x,T;q_0)$, which comes from Lemma \ref{lemma_positivity}, we have $q_1\le q_0$; while the smoothness of $q_1$ is ensured by Lemma \ref{lemma_smooth}. So we proved that $q_1\in\mathcal D$ and $q\le q_1\le q_0$. Applying Lemma \ref{lemma_monotone} again, it holds that 
 \begin{equation*}
  q=Kq\le Kq_1=q_2\le Kq_0=q_1\le q_0.
 \end{equation*}
Keeping on this argument, we conclude that $\{q_n\}_{n=0}^\infty$ is a decreasing sequence and has a lower bound $q$, which yields the pointwise convergence of $\{q_n\}_{n=0}^\infty$. We denote the limit by $\tilde q$ and obviously $q\le \tilde q\le q_0$. 

Next we need to prove $q=\tilde q$. From the monotone convergence theorem, we have that $\|q_n-\tilde q\|_{L^2(0,1)}\to 0$. With triangle inequality, we have 
\begin{align*}
 \|K\tilde q-\tilde q\|_{L^2(0,1)}&\le \|K\tilde q-Kq_n\|_{L^2(0,1)}+\|Kq_n-\tilde q\|_{L^2(0,1)}\\
 &=\|K\tilde q-Kq_n\|_{L^2(0,1)} +\|q_{n+1}-\tilde q\|_{L^2(0,1)}\\
 &=:I_1+I_2.
\end{align*}
It is obvious that $I_2\to 0$ as $n\to \infty$. For $I_1$, with Lemma \ref{lemma_positivity} we have 
\begin{align*}
 I_1=\|[\partial_t u(x,T;\tilde q)-\partial_t u(x,T;q_n)]/g'(x)\|_{L^2(0,1)}
 \le C \|\partial_t u(x,T;\tilde q)-\partial_t u(x,T;q_n)\|_{L^2(0,1)},
\end{align*}
where the constant $C$ depends on $m$. Setting $w=\partial_t u(x,t;\tilde q)-\partial_t u(x,t;q_n)$, we have 
\begin{equation*}
 \begin{cases}
  \begin{aligned}
   (\partial_t-\partial_x^2+\tilde q\partial_x+C_p)w(x,t)
   &=(q_n-\tilde q)\partial_{xt}^2 u(x,t;q_n)\ge0, &&(x,t)\in(0,1)\times(0,T];\\
   w_x(0,t)&=0,&&t\in(0,T];\\
   w_x(1,t)&=0,&&t\in(0,T];\\
   w(x,0)&=(q_n-\tilde q)v'\ge0,&&x\in(0,1).\\
  \end{aligned}
 \end{cases}
\end{equation*}
Then from Lemmas \ref{lemma_regularity_parabolic} and \ref{lemma_regularity}, it holds that 
\begin{align*}
 \|w(\cdot,T)\|_{L^2(0,1)}&\le C\|v'(q_n-\tilde q)\|_{L^2(0,1)}+ C\|\partial_{xt}^2 u(x,t;q_n)(q_n-\tilde q)\|_{L^2((0,1)\times(0,T))}\\
 &\le C\|q_n-\tilde q\|_{L^2(0,1)},    
\end{align*}
where the last constant $C$ depends on $b_1'(t)$, $b_2'(t)$ and $v'(x)$. So $I_1\le C\|q_n-\tilde q\|_{L^2(0,1)}$, which converges to zero as $n\to \infty$. To sum up, we could deduce that $\|K\tilde q-\tilde q\|_{L^2(0,1)}=0$, which gives that $\tilde q$ is a fixed point of $K$. Note that the pointwise convergence only ensures that $\tilde q\in L^\infty(0,1)$. Lemma \ref{lemma_smooth} and Assumption \ref{assumption} yield that $K\tilde q\in C^1([0,1])$. From the result $\|K\tilde q-\tilde q\|_{L^2(0,1)}=0$, we could pick the smooth version of $\tilde q$ and it gives that $\tilde q\in C^1$. 
Now we have proved that $q$ and $\tilde q$ are both fixed points of $K$ in $\mathcal D$ and they satisfy $q\le \tilde q$. Applying Lemma \ref{lemma_uniqueness}, we have $q=\tilde q$. 

The uniqueness could be deduced from above result and Lemma \ref{lemma_equivalence}. Given two solutions $q$ and $\hat q$ in $\mathcal D$ of inverse problem \eqref{inverse_problem}, it holds that $q_n\to q$ and $q_n\to \hat q$, which leads to $q=\hat q$ and completes the proof. 
\end{proof}

\section{Numerical reconstructions.}\label{section_num}
In this section we would solve inverse problem \eqref{inverse_problem} numerically, using the iteration \eqref{iteration} and Theorem \ref{theorem_unique}. 

\subsection{Discretization method.}
Numerical methods are designed to effectively reconstruct the drift term. 
We consider a backward Euler scheme for the model (\ref{PDE}). And, traditional finite difference methods are used for space discretization. 

Some notations are defined. Let the step of space $h=\frac{1}{M}$, $x_i=ih$, $0\le i\le M$, the step of time $\tau=\frac{T}{N}$, $t_n=n\tau$, $0\le n\le N$. For a grid function $u=\{u_i|0\le i\le M\}$, denote 
$$
Lu_i=\left\{
\begin{aligned}
\frac{1}{h}(u_{1}-u_{0})&,~~i=0,\\
\frac{1}{\tau}u_i&,~~1\le i\le M-1,\\  
\frac{1}{h}(u_{M}-u_{M-1})&,~~i=M.
\end{aligned}
\right.~~
\delta_x u_i=\left\{
\begin{aligned}
0&,~~i=0,\\
\frac{1}{2h}(u_{i+1}-u_{i-1})&,~~1\le i\le M-1,\\  
0&,~~i=M.
\end{aligned}
\right.
$$

$$
Iu_i=\left\{
\begin{aligned}
0&,~~i=0,\\
u_i&,~~1\le i\le M-1,\\  
0&,~~i=M.
\end{aligned}
\right.~~
\delta_x^2 u_i=\left\{
\begin{aligned}
0&,~~i=0,\\
\frac{1}{h^2}(u_{i+1}-2u_{i}+u_{i-1})&,~~1\le i\le M-1,\\  
0&,~~i=M.
\end{aligned}
\right.
$$

Let $u_i^n$ be the numerical solution at mesh point $(x_i,t_n)$, we have the discrete scheme:
\begin{align}\label{ds}
(L-\delta_x^2+q(x_i)\delta_x+C_pI)u_i^n=\frac{1}{\tau}Iu_i^{n-1}+\hat{f}(x_i,t_n),~~0\le i\le M,~1\le n\le N,
\end{align}
where $\hat{f}(x_i,t_n)=f(x_i)$, $1\le i\le M-1$, $\hat{f}(x_0,t_n)=b_1$ and $\hat{f}(x_M,t_n)=b_2(t_n)$.

\subsection{Mollification of noise data.}
In the previous section, theoretical results are deduced on the basis of noise-free measurement $g(x):=u(x,T)$.
However, most of the time the measurement is a set of noise data. It needs to be pre-processed before using the iterative algorithm to find the drift term.

Suppose that the data sampling points $\{x_i\}_{i=1}^K$ are evenly distributed in space $[0,1]$. Under the assumption of normal distribution, the noise of measurement ${g^\delta(x_i)}:=u(x_i,T)+e_i$ satisfies $\mathbb{E}[e_i]=0$, $\mathbb{E}[e_i^2]=\delta^2$. 
To denoise the noise data using the regularization method, we try to find $g^*$, which satisfies 
\begin{align}\label{reg-1}
\min\limits_{g} \|Ag-g^\delta\|^2 + \lambda\|\Gamma g\|^2, ~~\lambda>0,   
\end{align}
where $\|g\|=\sqrt{\sum_{i=0}^{K+1}g^2(x_i)}$, $A$ is the design matrix and $\Gamma$ is a the regularization matrix. 

We introduce the details of the regularization method (\ref{reg-1}) to technical instruct the numerical experiments. 
Combining with the boundary conditions in (\ref{PDE}), the design matrix is as follows:
\[
A=\begin{bmatrix}
-1 &    1&       &   &    \\
  &    1 &       &   &    \\
  &      &   \ddots  &     \\
  &      &       &   1&   \\
  &      &       &  -1&  1
\end{bmatrix}_{K,K}.
\]
The corresponding measurements $g$ and $g^\delta$ are given:
\[
g=\begin{bmatrix}
g(x_0)      \\
g(x_1)      \\
\vdots      \\
g(x_M)      \\
\end{bmatrix}_{K,1},~~~
\tilde{g}^\delta=\begin{bmatrix}
 hb_1(T)   \\
g^\delta(x_1)  \\
\vdots      \\
 hb_2(T) 
\end{bmatrix}_{K,1}.
\]
The regularization matrix $\Gamma$ is introduced as follows:
\[
\Gamma=\frac{1}{(K-1)^2}\begin{bmatrix}
1  &    -2 &      1    &   &    \\
  &      &   \ddots    &     \\
  &      &       1&   -2&  1 
\end{bmatrix}_{K-2,K}.
\]
The point of interest is to find the optimal regularization parameter $\lambda$. The details of the choice for it refer to \cite{CenZhang:2025,ChenTuoZhang:2018}. It is worth noting that data mollification is important for numerical inversion, see Figures \ref{ex3_d_e_w} and \ref{ex3_e_w} for Ex. \ref{ex3}.

\subsection{Numerical experiments.}
From the iteration method (\ref{iteration}), the corresponding discrete form of it is as following:
\begin{equation}\label{iteration-ds}
 q_0(x_i)=[f(x_i)+\delta_x^2g(x_i)-C_pg(x_i)]/\delta_xg(x_i),~~ 1\le i\le M-1,
\end{equation}
when $i=0$ or $M$, $g_{x}$ and $g_{xx}$ are obtained by the linear interpolation based on the results of nearby inner points.
Then $q_{n+1}$ is given from $Kq_n$, and $\partial_tu(x_i,T;q_n)$ in the iterative operator is $\frac{u(x_i,T;q_n)-u(x_i,T-\tau;q_n)}{\tau}$. 

\begin{example}\label{ex1}
The source, initial function, boundary conditions, the potential and the finial time are set: $T=1$, $C_p=5$, $v(x)=\sin(\pi x)$, $b_1(t)=1$, $b_2(t)=1+t$, $f(x)=10+10x$. We try to reconstruct the following smooth drift functions:
$$(a)~q=\sin(x),~~(b)~q=\begin{cases}x^2, &0\leq x\leq \frac12\\
-x^2+2x-\frac12, &\frac12 < x\leq 1\end{cases}.$$ 
\end{example}
The case of noise-free measurement $g(x)$ is considered. The number of grids are $M=N=100$ for the numerical scheme (\ref{ds}). The number of observation data is $M$. Our aim is to check the effectiveness of our iterative method (\ref{iteration}).  The numerical results are presented in Figures \ref{ex1_ab_q} and \ref{ex1_ab_q_1}.  It shows that one only needs to solve the forward problem twice to get a satisfied drift term. 

\begin{figure}[h!]
		\centering
		\setlength{\abovecaptionskip}{0.2cm}
        \includegraphics[width=0.7\linewidth]{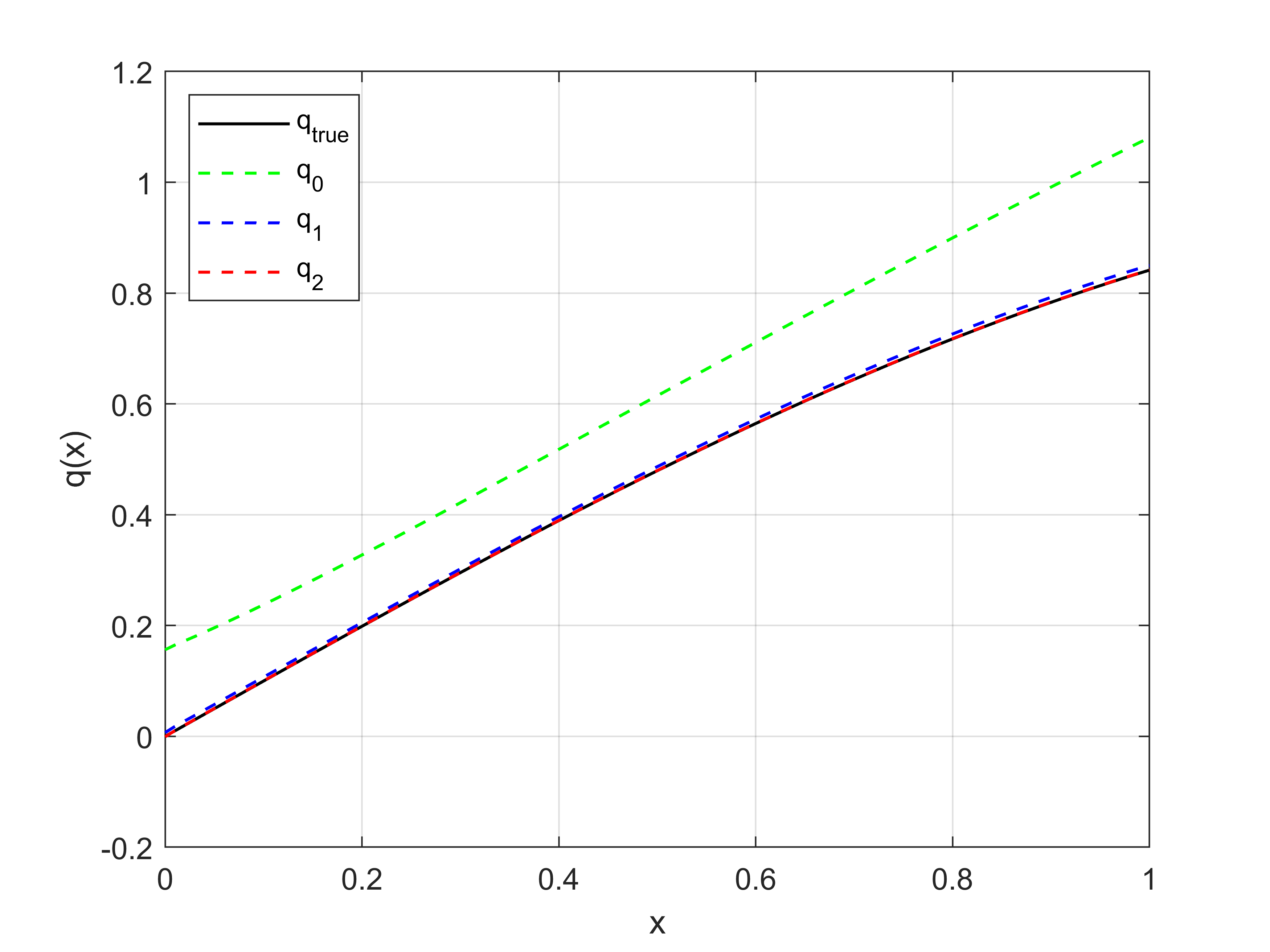}      
         \caption{Recover $q(x)$ of Ex. \ref{ex1} (a).}
         \label{ex1_ab_q}
\end{figure}

Drifts (a) and (b) in Ex. \ref{ex1} exhibit different regularity. While $q\in C^\infty(0,1)$ for drift (a), the derivative is not continue at $x=\frac12$ for drift (b).

\begin{figure}[h!]
		\centering
		\setlength{\abovecaptionskip}{0.2cm}
        \includegraphics[width=0.7\linewidth]{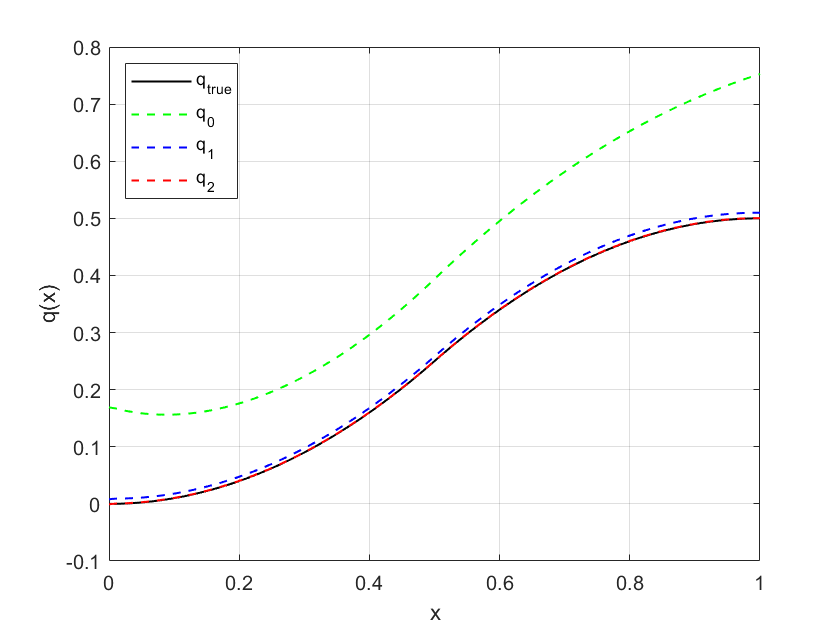}      
         \caption{Recover $q(x)$ of Ex. \ref{ex1} (b).}
         \label{ex1_ab_q_1}
\end{figure}

\begin{example}\label{ex1-vs}
The finial time $T=0.5$. The source, initial function, boundary conditions, and potential are the same as in Example \ref{ex1}. We also consider the case of noise-free measurement $g(x)$. The following drift term with singularity points is recovered.
$$(c)~q=\begin{cases}x, &0\leq x\leq \frac12\\
1-x, &\frac12 < x\leq 1\end{cases},~~(d)~q=\begin{cases}20|x-\frac{1}{10}|-1, &0\leq x\leq \frac15\\
20|x-\frac{3}{10}|-1, &\frac15< x\leq \frac25\\
20|x-\frac{1}{2}|-1, &\frac25< x\leq \frac35\\
20|x-\frac{7}{10}|-1, &\frac35< x\leq \frac45\\
20|x-\frac{9}{10}|-1, &\frac45< x\leq 1\end{cases}.$$
\end{example}

Despite the derivative discontinuity in the middle for drift (c), the method achieves inversion. Furthermore, it also could present good inversion result by the iterative method for periodic absolute value function (d).

\begin{figure}[h!]
		\centering
		\setlength{\abovecaptionskip}{0.2cm}
        \includegraphics[width=0.7\linewidth]{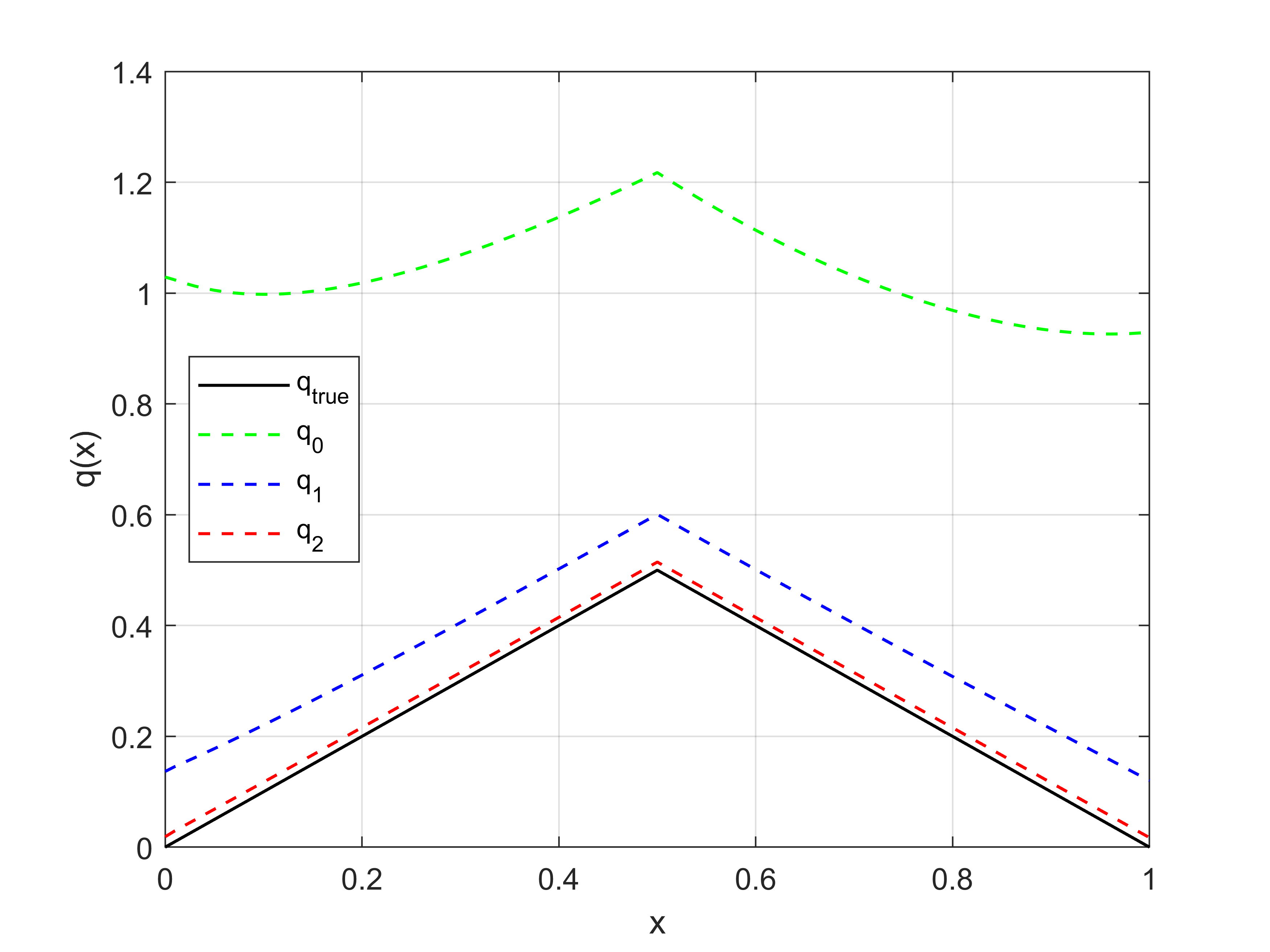}      
         \caption{Recover $q(x)$ of Ex. \ref{ex1-vs} (c).}
         \label{ex2_ab_q_vs}
\end{figure}

\begin{figure}[h!]
		\centering
		\setlength{\abovecaptionskip}{0.2cm}
        \includegraphics[width=0.7\linewidth]{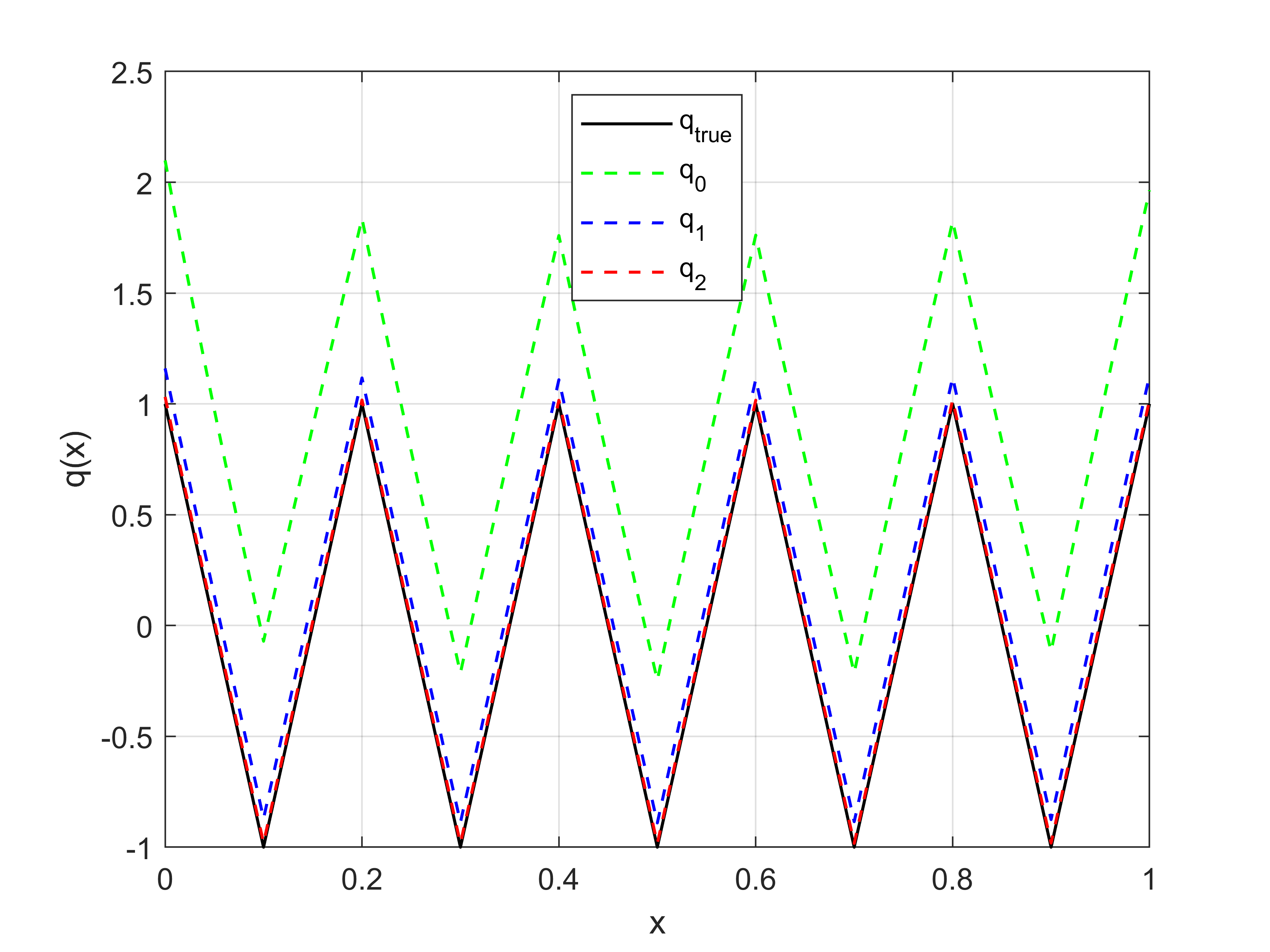}      
         \caption{Recover $q(x)$ of Ex. \ref{ex1-vs} (d).}
         \label{ex2_ab_q_vs_1}
\end{figure}

\begin{example}\label{ex3}
The source, boundary conditions, potential, and finial time are the same as in Example \ref{ex1}. More general drift terms $q$ are considered.
$$
(e)~q=\begin{cases}0, &0\leq x\leq \frac14\\
1, &\frac14 < x\leq \frac12\\
0, &\frac12 < x\leq \frac34\\
1, &\frac34 < x\leq 1\\
\end{cases},~~(f)~q=\begin{cases}-1, &0\leq x< 0.2\\
0.5x, &0.2 \le x\leq 0.8\\
-1, &0.8 < x\leq 1\\
\end{cases}.$$ 
For case (e), $q\in L^\infty([0,1])$ is a staircase function. Two kinds of measurements $g^\delta(x)$ with different noise levels $\delta=1\%$, $3\%$ are used for inversion.
\end{example}

The number of grids are $N=80$, $M=20$ for the numerical scheme (\ref{ds}). The number of observation data is $K=10^7$. We use the iteration method (\ref{iteration}) to recover the drift terms. Figure \ref{ex3_d} shows that our method is still effective for noise data after mollification.

\begin{figure}[h!]
		\centering
        \begin{minipage}{0.45\linewidth}
		\setlength{\abovecaptionskip}{0.2cm}
        \includegraphics[width=1\linewidth]{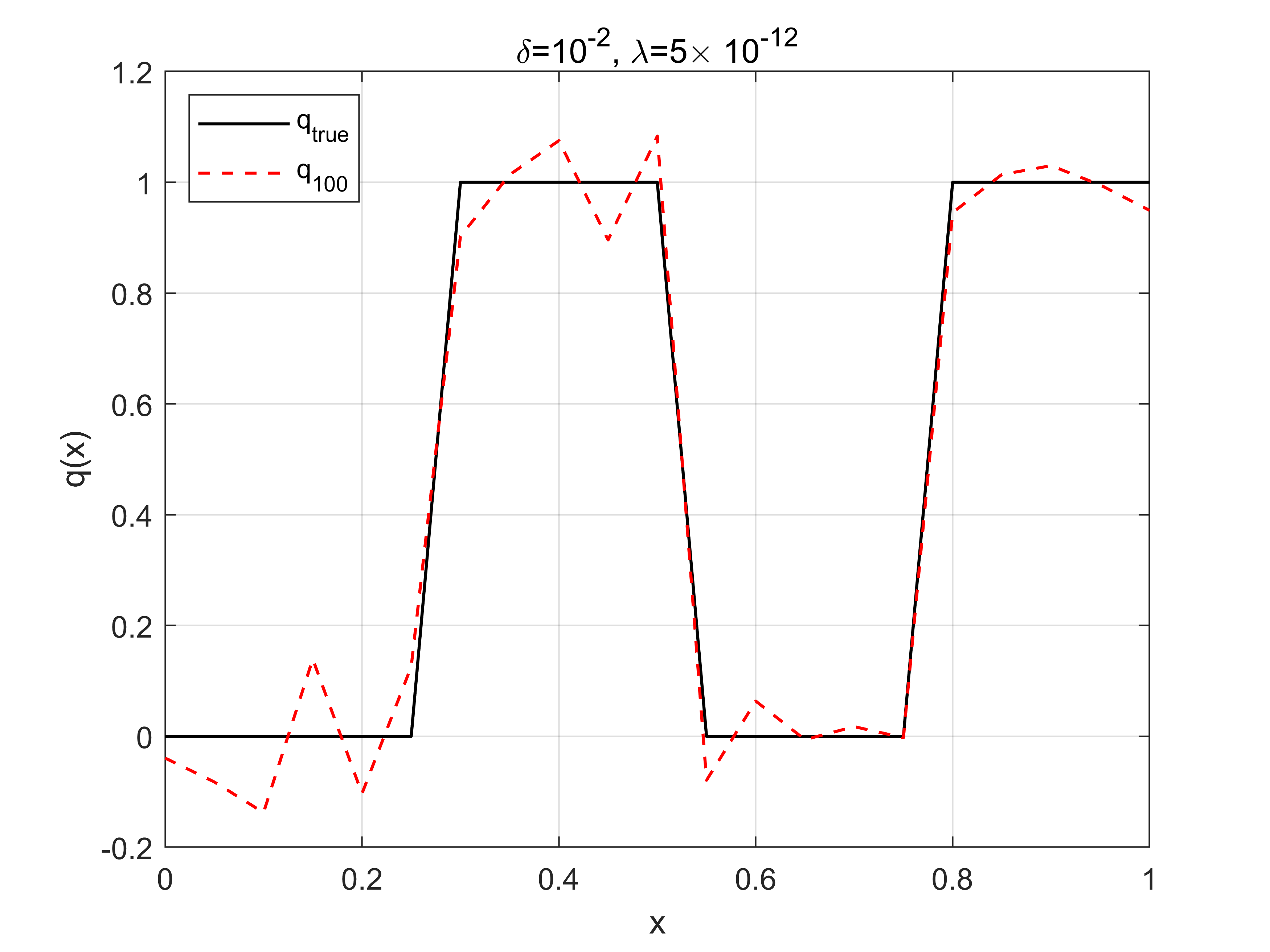}      
	    \end{minipage}
        \begin{minipage}{0.45\linewidth}
		\setlength{\abovecaptionskip}{0.2cm}
        \includegraphics[width=1\linewidth]{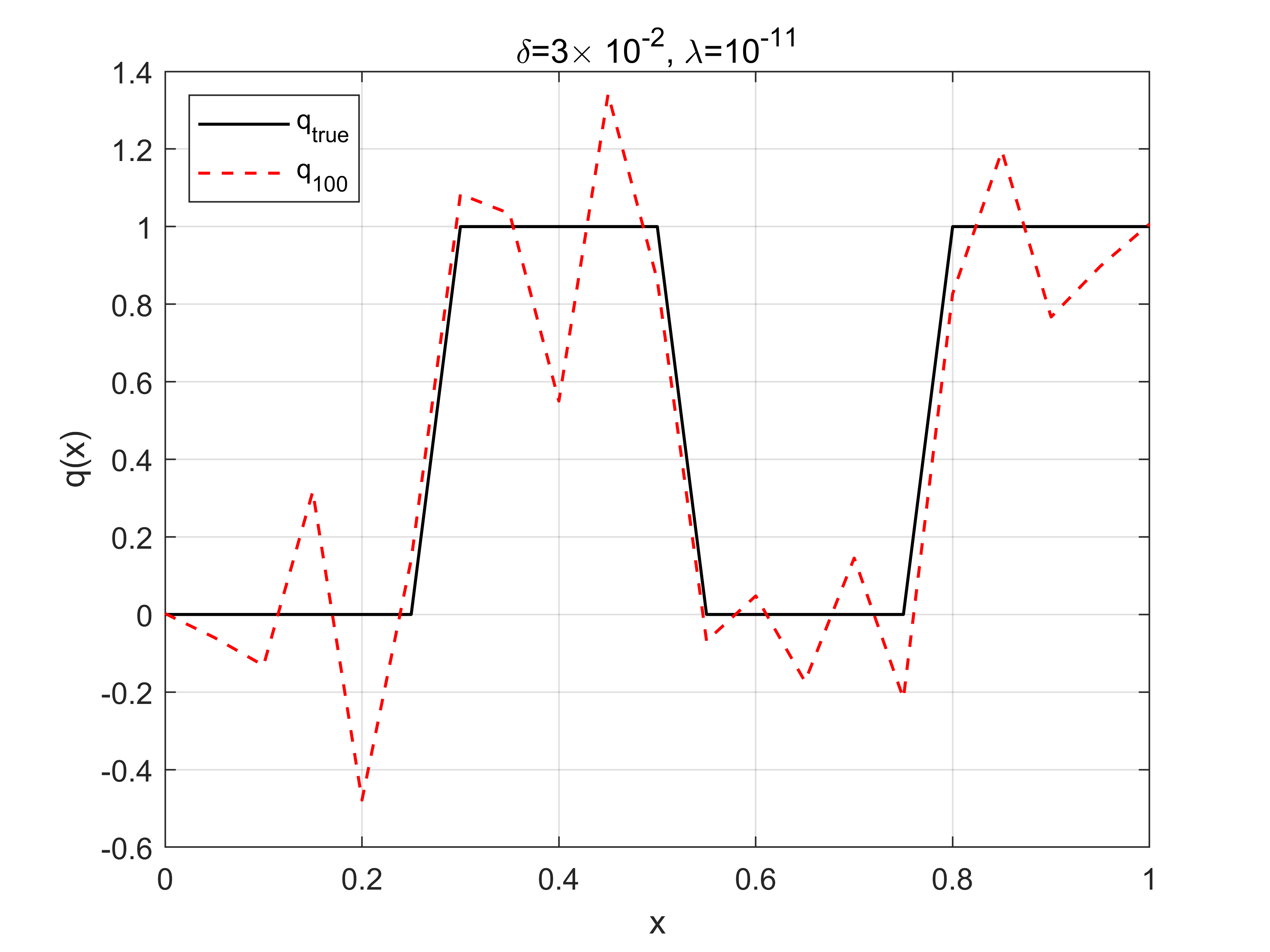}
	    \end{minipage}
         \caption{Recover $q(x)$ of Ex. \ref{ex3} (e) with noise, [1\%, left] and [3\%, right].}
         \label{ex3_d}
\end{figure}

\begin{figure}[h!]
		\centering
        \begin{minipage}{0.45\linewidth}
		\setlength{\abovecaptionskip}{0.2cm}
        \includegraphics[width=1\linewidth]{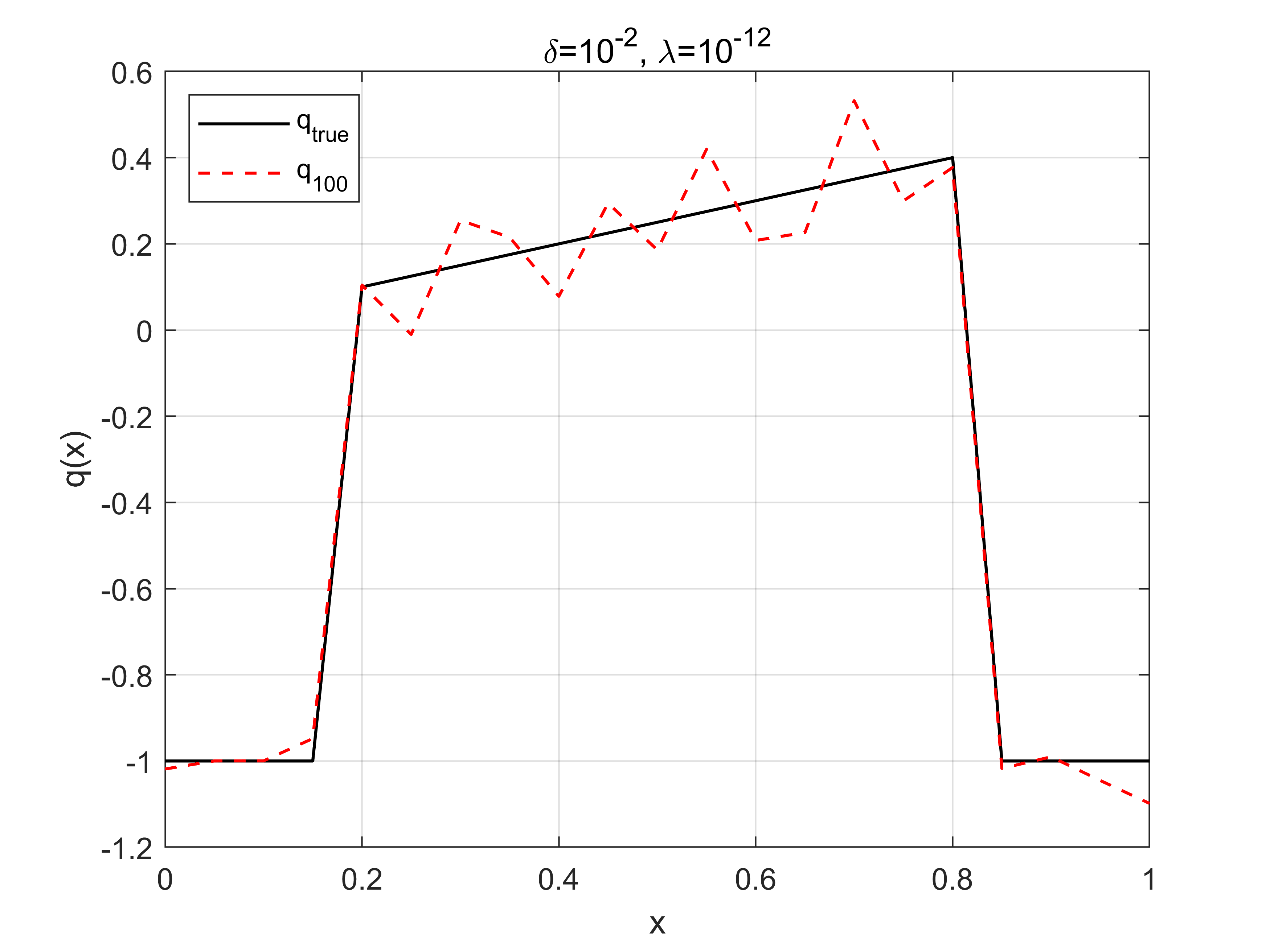}      
	    \end{minipage}
        \begin{minipage}{0.45\linewidth}
		\setlength{\abovecaptionskip}{0.2cm}
        \includegraphics[width=1\linewidth]{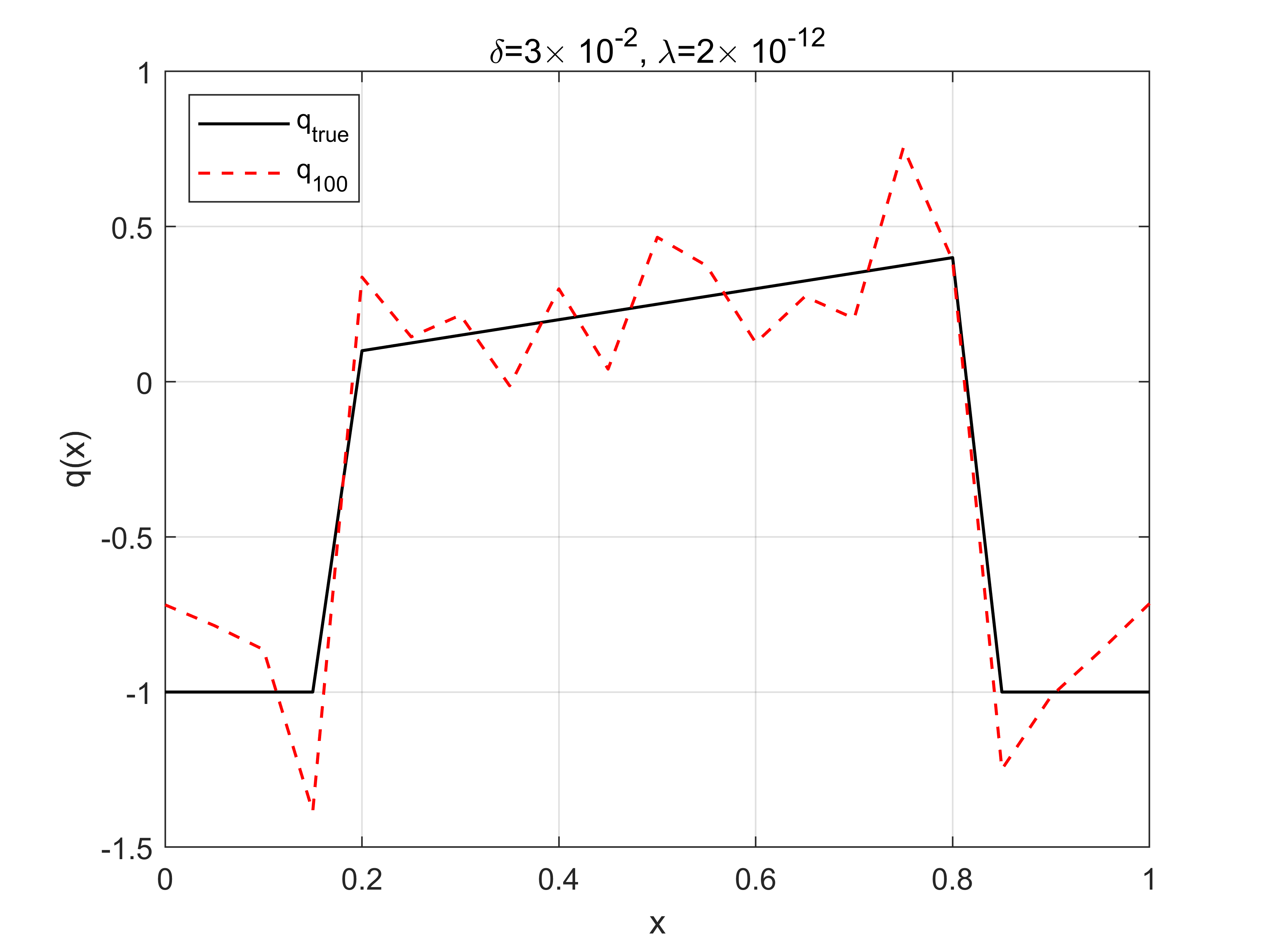}
	    \end{minipage}
         \caption{Recover $q(x)$ of Ex. \ref{ex3} (f) with noise, [1\%, left] and [3\%, right].}
         \label{ex3_e}
\end{figure}

\begin{figure}[h!]
		\centering
        \begin{minipage}{0.45\linewidth}
		\setlength{\abovecaptionskip}{0.2cm}
        \includegraphics[width=1\linewidth]{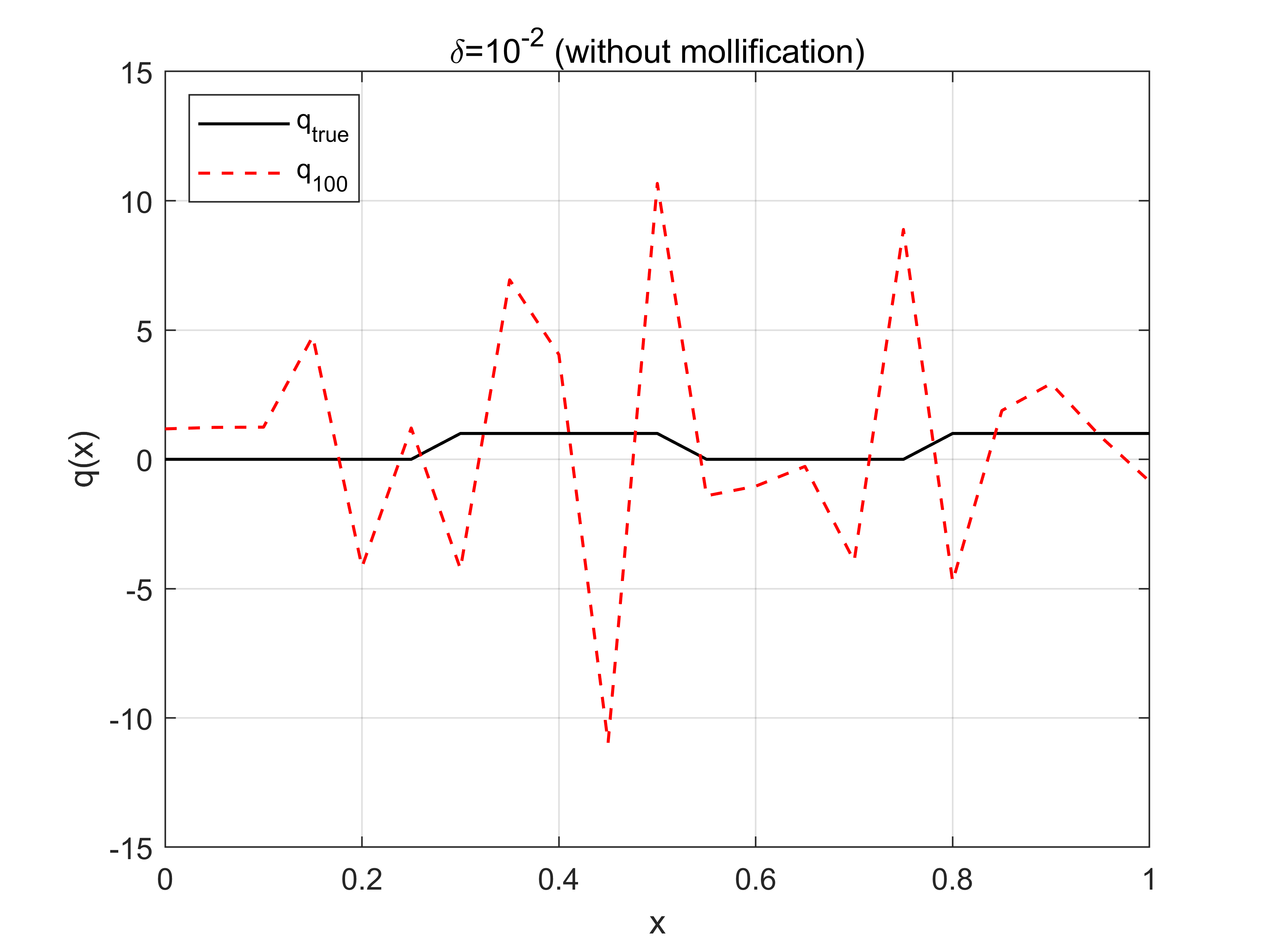}      
	    \end{minipage}
        \begin{minipage}{0.45\linewidth}
		\setlength{\abovecaptionskip}{0.2cm}
        \includegraphics[width=1\linewidth]{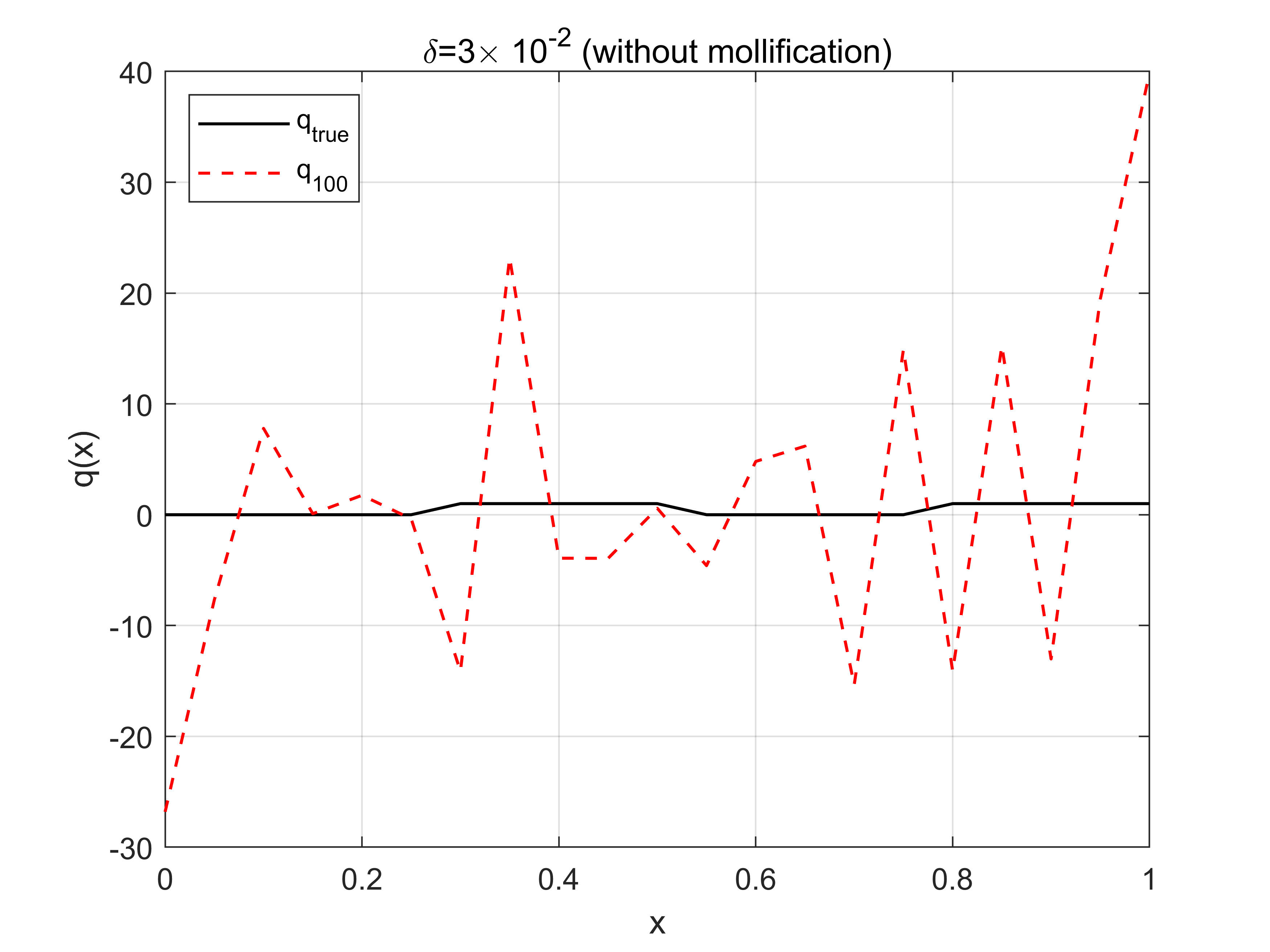}
	    \end{minipage}
        \caption{Recover $q(x)$ of Ex. \ref{ex3} (e) without mollification, [1\%, left] and [3\%, right].}
        \label{ex3_d_e_w}
\end{figure}

It is worth noting that data mollification is important for numerical inversion. Noisy data would seriously undermine the stability of numerical inversion, see Figures \ref{ex3_d_e_w} and \ref{ex3_e_w}.

\begin{figure}[h!]
		\centering
        \begin{minipage}{0.45\linewidth}
		\setlength{\abovecaptionskip}{0.2cm}
        \includegraphics[width=1\linewidth]{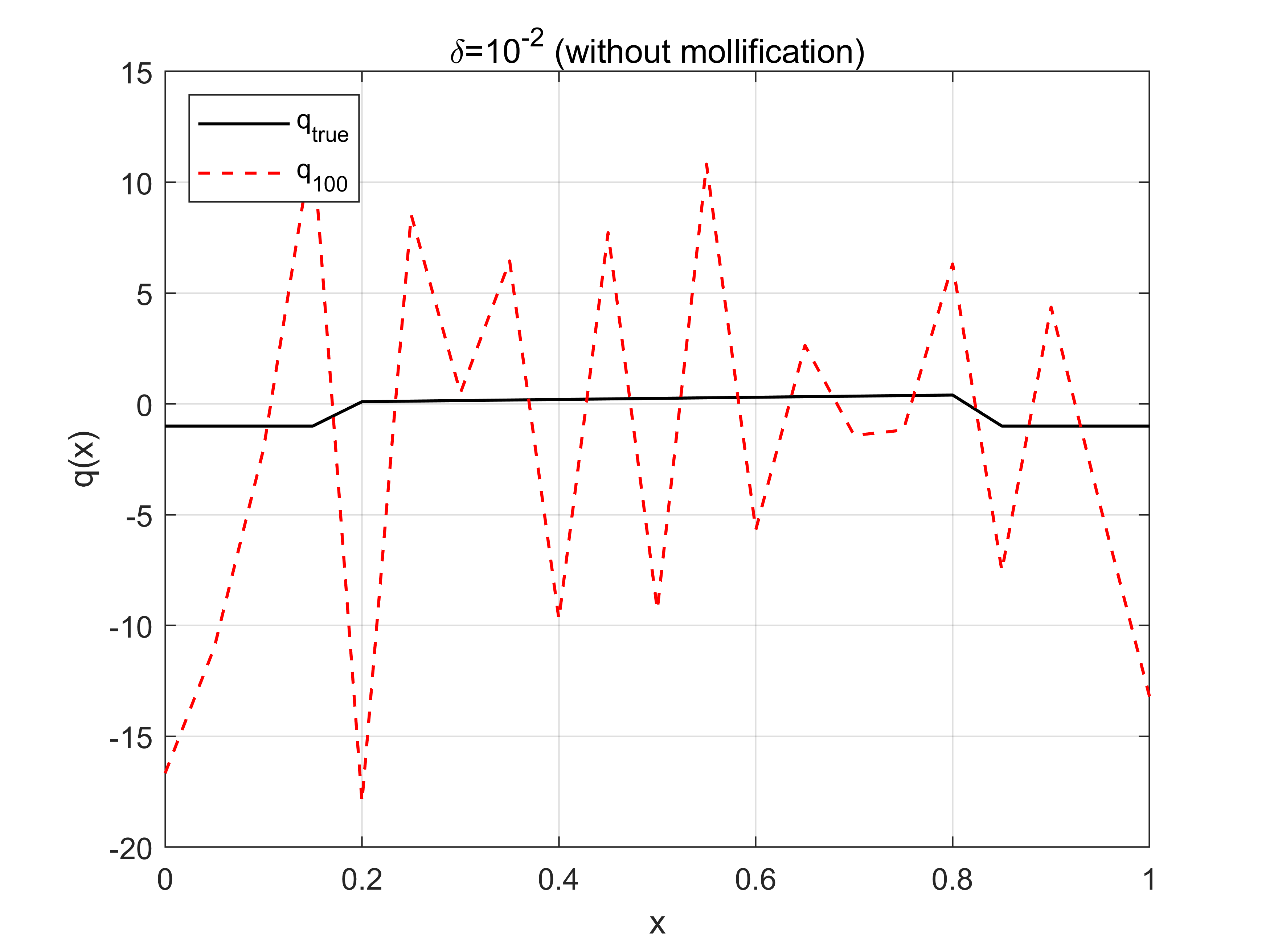}      
	    \end{minipage}
        \begin{minipage}{0.45\linewidth}
		\setlength{\abovecaptionskip}{0.2cm}
        \includegraphics[width=1\linewidth]{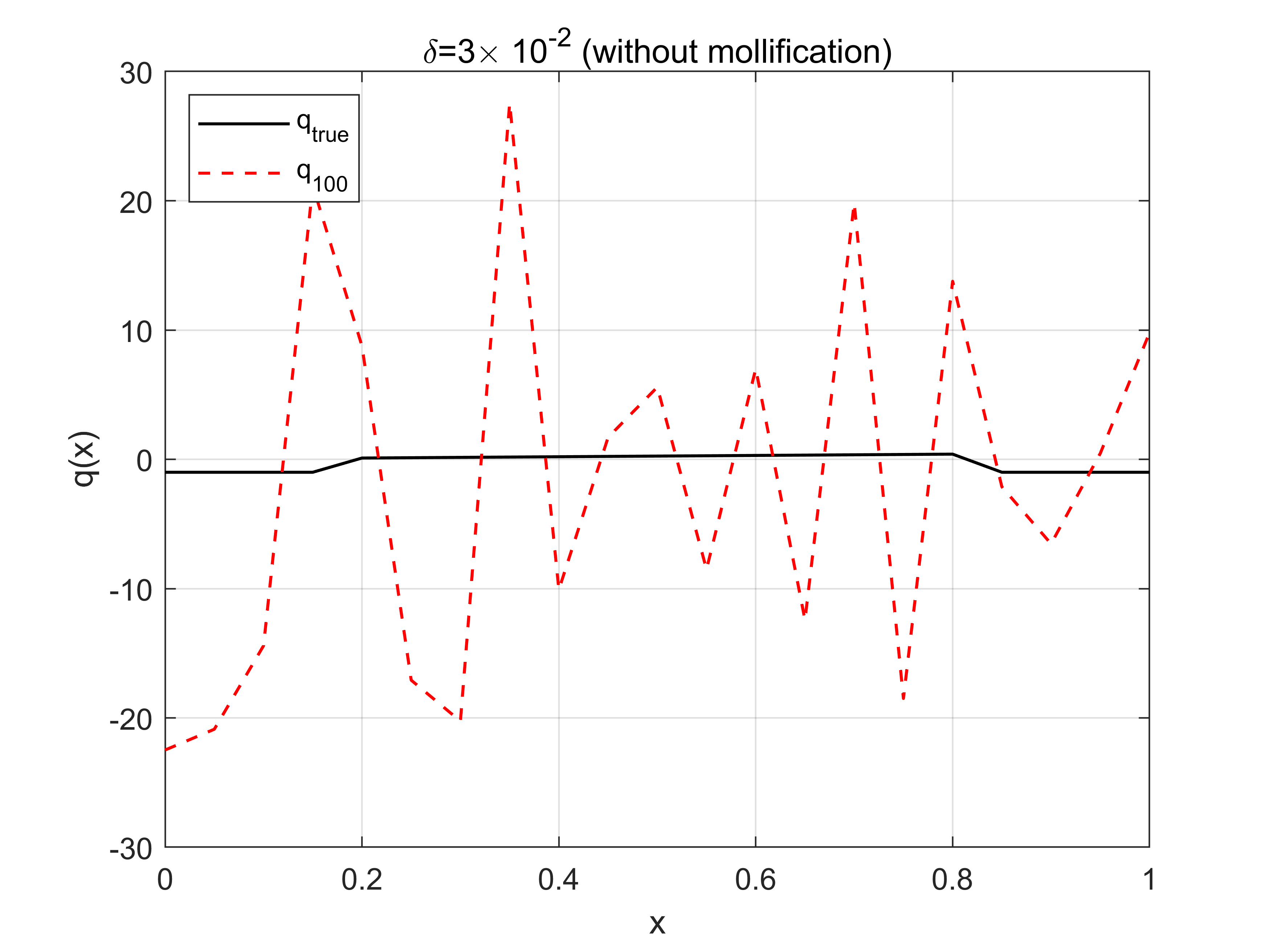}
	    \end{minipage}
        \caption{Recover $q(x)$ of Ex. \ref{ex3} (f) without mollification, [1\%, left] and [3\%, right].}
        \label{ex3_e_w}
\end{figure}

\section{Concluding remarks and future works.}\label{section_con}
In this manuscript, we solve the inverse drift problem in the one-dimensional parabolic equation by the monotone operator method. We first construct an operator from the equation and the data, whose fixed points are the desired drift. Then some properties of the operator are proved, such as equivalence and monotonicity, which are applied to prove the uniqueness theorem. We note that the proof of uniqueness theorem contains an iteration which converges to the solution of inverse problem. This leads to the iterative algorithm for solving the inverse problem numerically. To handle the ill-posedness of the inverse problem, we first add the mollification on the data and then take several experiments. The satisfactory numerical results indicate the effectiveness of the used algorithm.       

In the future, we would investigate the inverse drift problem in high-dimensional case, which should more challenging. In high-dimensional situation, the drift is the coefficient of the gradient of solution, which is represented as $q\cdot\nabla u$. We could see now the drift $q$ is a vector field. So we should overcome the difficulties as how to construct the operator $K$, how to control the boundary, initial conditions and source term to make sure the well-definedness of the operator, and so on.    

\section*{Acknowledgments.}
Zhidong Zhang is supported by the National Key Research and Development Plan of China (Grant No. 2023YFB3002400). Wenlong Zhang is partially supported by the National Natural Science Foundation
of China under grants 12371423 and 12241104.

\end{document}